\documentclass[11pt]{article}
\usepackage{amscd, amsmath, amssymb, amsthm}
\usepackage[all,cmtip]{xy}
\usepackage[pagebackref]{hyperref}
\usepackage{enumerate} 

\title{Endomorphisms of Fano 3-folds and log Bott vanishing}
\author{Burt Totaro}
\date{  }

\def\R{\text{\bf R}}
\def\C{\text{\bf C}}
\def\P{\text{\bf P}}

\def\arrow{\rightarrow}

\DeclareMathOperator{\sing}{sing}

\DeclareMathOperator{\Nef}{Nef}
\DeclareMathOperator{\tor}{tor}
\DeclareMathOperator{\cotor}{cotor}
\DeclareMathOperator{\Tor}{Tor}
\DeclareMathOperator{\Gr}{Gr}
\DeclareMathOperator{\tr}{tr}

\begin{document}
\maketitle
\newtheorem{theorem}{Theorem}[section]
\newtheorem{corollary}[theorem]{Corollary}
\newtheorem{proposition}[theorem]{Proposition}
\newtheorem{lemma}[theorem]{Lemma}

\theoremstyle{definition}
\newtheorem{definition}[theorem]{Definition}
\newtheorem{example}[theorem]{Example}

\theoremstyle{remark}
\newtheorem{remark}[theorem]{Remark}

A basic problem of algebraic dynamics is to determine which projective
varieties admit an endomorphism of degree greater than 1.
To avoid degenerate cases, we focus on
{\it int-amplified }endomorphisms $f\colon X\to X$, meaning that
there is an ample divisor $H$ such that $f^*H-H$ is ample
\cite{Meng-amplified, MZnormal}. Fakhruddin, Meng, Zhang, and Zhong
conjectured that a smooth complex projective rationally connected variety
that admits an int-amplified endomorphism
must be a toric variety \cite[Question 4.4]{Fakhruddin},
\cite[Question 1.1]{MZZ}. Kawakami
and I introduced a new approach to this problem, showing
that a variety with an int-amplified
endomorphism must satisfy {\it Bott vanishing}
\cite[Theorem C]{KTot}.

In this paper, we extend that result to
a logarithmic version of Bott vanishing for an endomorphism
with a totally invariant divisor (Theorem \ref{endo-log-Bott}).
We also connect log Bott vanishing with some related problems:
which varieties are images of toric varieties (Theorem
\ref{toric-log-Bott})? Which varieties
admit morphisms of unbounded degree from some other variety
(Theorem \ref{morphism-log-Bott})?

We apply these results to Fano 3-folds.
Meng, Zhang, and Zhong showed that the only smooth complex Fano 3-folds
that admit an int-amplified endomorphism are the toric ones
\cite[Theorem 1.4]{MZZ}. 
Also, Achinger, Witaszek, and Zdanowicz showed
that the only smooth complex Fano 3-folds that are images
of toric varieties are the toric ones
\cite[proof of Theorem 4.4.1]{AWZ1},
\cite[Theorems 6.9 and 7.7]{AWZ2}.
Using log Bott vanishing, we reprove both results
and extend them to characteristic $p$, for morphisms of degree
prime to $p$ (Theorems \ref{endo-3-fold}
and \ref{image-3-fold}).
This resolves \cite[Question 1.6]{KTot}.
For Fano 3-folds with Picard number 1,
these extensions already appeared in \cite[Theorem A
and Proposition 3.10]{KTot}.

\section{Notation}

A Weil divisor
(with integer coefficients) on a normal projective variety
is called {\it ample }if
some positive multiple is an ample Cartier divisor.
A {\it contraction }of a normal variety $X$
over a field $k$
is a proper morphism $\pi\colon X\to Y$ with
$\pi_*O_X=O_Y$. For a projective variety $X$ over $k$,
$N_1(X)$ is the vector space of 1-cycles with real coefficients
modulo numerical equivalence, which has finite dimension.
The {\it N\'eron-Severi} space $N^1(X)$ is the space
of $\R$-Cartier divisors modulo numerical equivalence,
and so $N^1(X)=N_1(X)^*$.

A morphism of varieties $f\colon Y\to X$ over a field $k$
is {\it separable }if it is dominant
and the function field $k(Y)$ is a separable field extension
of $k(X)$. For $k$ algebraically closed, $f$ is separable if and only if
the derivative of $f$ is surjective at some smooth point of $Y$.
For an endomorphism $f$ of a variety $X$,
a closed subset $S$ of $X$ is {\it totally invariant }under $f$
if $f^{-1}(S)=S$.

For a normal variety $X$ over a field $k$ and
$i\geq 0$, we write $\Omega^i_X$ for $\Omega^i_{X/k}$. The sheaf
of {\em reflexive differentials} $\Omega^{[i]}_X$
is the double dual $(\Omega^i_X)^{**}$.
For a reduced divisor $D$ on $X$, $\Omega^{[i]}_X(\log D)$
(reflexive differentials with {\it log poles }along $D$)
denotes the sheaf of $i$-forms $\alpha$ on the smooth locus
of $X-D$
such that both $\alpha$ and $d\alpha$ have at most a simple pole
along each component of $D$. For a reflexive sheaf $M$ and
a Weil divisor $E$ on $X$, we write $M(E)$
for the reflexive sheaf $(M\otimes O_X(E))^{**}$.
If $X$ is smooth over $k$, then $O_X(E)$ is a line bundle,
and $M(E)$
is just the tensor product $M\otimes O_X(E)$.

\section{The trace of a differential form with log poles}

Here is the key lemma for this paper,
extending the proof of \cite[Theorem C]{KTot} to allow log poles.

\begin{lemma}
\label{trace}
Let $f\colon Y\to X$ be a finite surjective morphism of normal
varieties over a perfect field $k$. Let $A$ be a Weil divisor
on $X$. Let $E_X$ and $D_X$ be reduced divisors on $X$
with $0\leq E_X\leq D_X$. Let $D_Y$ be the sum of the
components of $f^{-1}(D_X)$ along which the ramification
degree $e$ of $f$ is invertible in $k$, and let $E_Y$ be the corresponding
divisor inside $f^{-1}(E_X)$.
Then, for every $i\geq 0$,
the pullback and pushforward of differential forms give maps
of reflexive sheaves
$$\Omega^{[i]}_X(\log D_X)(A-E_X)\to f_*(\Omega^{[i]}_Y(\log D_Y)(f^*A-E_Y))
\to \Omega^{[i]}_X(\log D_X)(A-E_X),$$
with composition equal to multiplication by $\deg(f)$.
\end{lemma}

In characteristic $p$, even if the morphism $f$
in Lemma \ref{trace} has degree prime to $p$,
it may be wildly ramified along some divisors.
That requires the careful choice of $D_Y$ and $E_Y$
in the statement, in order to get something true.

\begin{proof}
Since $k$ is perfect, the normal varieties $X$ and $Y$ over $k$
are geometrically normal \cite[Tag 038O]{Stacks},
hence smooth over $k$ in codimension 1.

As in the statement, let $D_Y$ be the sum of the
components of $f^{-1}(D_X)$ along which the ramification
degree $e$ of $f$ is invertible in $k$. (The {\it ramification
degree }$e$
along a component $D_2$ of $f^{-1}(D_X)$
is the coefficient of $D_2$ in the pullback
Weil divisor $f^*D_X$, which is defined since $f$ is finite
and $D_X$ is generically Cartier.
The condition that $e$ is invertible in $k$
is not the same as $f$ being
tamely ramified along $D_2$; tame ramification
would mean that $e$ is invertible in $k$ {\it and }that
$f\colon D_2\to f(D_2)$ is separable.)
Likewise, let $E_Y$ be the sum of the components of $f^{-1}(E_X)$
along which the ramification degree of $f$ is invertible in $k$.

We want to construct pullback and pushforward maps of reflexive sheaves:
$$\Omega^{[i]}_X(\log D_X)(-E_X)\to f_*\Omega^{[i]}_Y(\log D_Y)(-E_Y)
\to \Omega^{[i]}_X(\log D_X)(-E_X).$$
First consider the pullback map.
Since we are mapping to a reflexive sheaf,
it suffices to define this map outside the subset
$X^{\sing}\cup f(Y^{\sing})$, which has codimension
at least 2 in $X$. The map is the usual pullback of differential forms
outside $D_X$. So it suffices to show that the pullback
sends $\Omega^i_X(\log D_X)(-E_X)$ into $\Omega^i_Y(\log D_Y)(-E_Y)$
near the generic point of each component $D_2$ of $f^{-1}(D_X)$.
Let $D_1=f(D_2)$. 

There are four cases.
\begin{description}
\item[Case 1] $D_1$ not in $E_X$,
$e\in k^*$. Then we need to show that $\Omega^i_X(\log D_1)$ pulls back
into $\Omega^i_Y(\log D_2)$ near the generic point of $D_2$.
Let $g$ be a local defining function of $D_1$ in $X$ and $h$ a local defining
function of $D_2$ in $Y$. We have $f^*(g)=h^eu$ for some unit
$u\in O_{Y,D_2}^*$. Then $f^*(\frac{dg}{g})=e\frac{dh}{h}+\frac{du}{u}$,
and so it is clear that forms in
$\Omega^i_X(\log D_1)=\Omega^i_X+\frac{dg}{g}\Omega^{i-1}_X$
pull back to $\Omega^i_Y(\log D_2)=\Omega^i_Y
+\frac{dh}{h}\Omega^{i-1}_Y$ near the generic point of $D_2$.
\item[Case 2] $D_1$ not in $E_X$, $e=0\in k$. Then we need to show
that forms in $\Omega^i_X(\log D_1)$ pull back to $\Omega^i_Y$ near
the generic point of $D_2$. This is clear by the formulas in Case 1,
using that $e=0\in k$.
\item[Case 3] $D_1$ in $E_X$, $e\in k^*$.
Then we need to show that $\Omega^i_X(\log D_1)(-D_1)$ pulls back
into $\Omega^i_Y(\log D_2)(-D_2)$
near the generic point of $D_2$. This follows from Case 1, using
that $g$ pulls back to a unit times a positive power of $h$.
\item[Case 4] $D_1$ in $E_X$, $e=0\in k$.
Then we need to show that $\Omega^i_X(\log D_1)(-D_1)$ pulls back
into $\Omega^i_Y$
near the generic point of $D_2$. This is clear, using that
$\Omega^i_X(\log D_1)(-D_1)$ is contained in $\Omega^i_X$.
\end{description}

Next, we want to define
the {\it pushforward }(or {\it trace}) map
$$f_*\Omega^{[i]}_Y(\log D_Y)(-E_Y) \to \Omega^{[i]}_X(\log D_X)(-E_X).$$
Outside $D_X\cup X^{\sing}\cup f(Y^{\sing})$, this trace map was defined
by Garel and Kunz \cite{Garel}, \cite[section 16]{Kunz},
\cite[Tag 0FLC]{Stacks}. (They assume that the finite morphism $f$
is flat with lci fibers; that holds on the open set mentioned,
since both $X$ and $Y$ are smooth there \cite[Tags 00R4 and 09Q7]{Stacks}.)
Since we are mapping into a reflexive sheaf, it remains
to check that the trace map above is regular
near the generic point of each component $D_1$
of $D_X$. It suffices to check this after replacing $X$ by an elementary
\'etale neighborhood \cite[Tag 02LE]{Stacks}
of the generic point of $D_1$, in such a way 
that $Y$ becomes a disjoint union of varieties containing
the different components $D_2$ of $f^{-1}(D_1)$. (Then the trace
from $Y$ to $X$ is the sum of the traces for these different varieties.)
We can work
on one of these varieties; that is, we can assume that $f^{-1}(D_1)$
is an irreducible divisor $D_2$.

There are four cases.
\begin{description}
\item[Case 1] $D_1$ not in $E_X$,
$e\in k^*$. Then we need to show that the pushforward
of a form in $\Omega^i_Y(\log D_2)$
lies in $\Omega^i_X(\log D_1)$ near the generic point
of $D_1$. We have $\Omega^i_Y(\log D_2)=\Omega^i_Y
+\frac{dh}{h}\Omega^{i-1}_Y$.
Since $f^*(g)=h^eu$ in the notation above,
we have $e\frac{dh}{h}=f^*(\frac{dg}{g})
-\frac{du}{u}$. Since $e\in k^*$, it follows that
$\Omega^i_Y(\log D_2)=\Omega^i_Y+f^*(\frac{dg}{g})\Omega^{i-1}_Y$.
By the projection formula \cite[Tag 0FLC]{Stacks},
the trace of these forms lies in
$\Omega^i_X+\frac{dg}{g}\Omega^{i-1}_X=\Omega^i_X(\log D_1)$.
\item[Case 2] $D_1$ not in $E_X$, $e=0\in k$.
Then we need to show
that forms in $\Omega^i_Y$ push forward
to $\Omega^i_X(\log D_1)$ near
the generic point of $D_1$. This is easy, since the pushforward
is contained in $\Omega^i_X\subset \Omega^i_X(\log D_1)$.
\end{description}

For cases 3 and 4, we use the following property of the trace map
on differential forms \cite[Remark 5.7]{Hubl}:

\begin{proposition}
\label{compatible}
Let $f\colon Y\to X$ be a finite flat morphism with lci fibers
over a field $k$. Suppose that $D_1$ is an irreducible divisor in $X$
such that $D_2:=f^{-1}(D_1)$ is also irreducible. Let $e$
be the ramification degree of $f$ along $D_2$. Then the following diagram
commutes near the generic point of $D_1$:
$$\xymatrix@C-10pt@R-10pt{
\Omega^*_Y \ar[r] \ar[d]^{\tr_f} & \Omega^*_{D_2}\ar[d]^{e\, \tr_f} \\
\Omega^*_X \ar[r] & \Omega^*_{D_1}.
}$$
\end{proposition}

In our situation (above), the flatness and lci assumptions of Proposition
\ref{compatible} hold near the generic points of $D_1$ and $D_2$,
because $X$ and $Y$ are smooth there.

\begin{description}
\item[Case 3] $D_1$ in $E_X$, $e\in k^*$.
Then we need to show that all the forms in
$\Omega^i_Y(\log D_2)(-D_2)$ push forward into
$\Omega^i_X(\log D_1)(-D_1)$ near the generic point of $D_1$.
That is, we want to show that the trace of
$h\, \Omega^i_Y+h\frac{dh}{h}\Omega^{i-1}_Y$
is contained in $g\, \Omega^i_X
+g\frac{dg}{g}\Omega^{i-1}_X$.
Proposition \ref{compatible}
gives that the trace of $(h,dh) \Omega^*_Y$ is contained
in $(g,dg) \Omega^*_X$, which proves the claim.
\item[Case 4] $D_1$ in $E_X$, $e=0\in k$. 
Then we need to show that forms in $\Omega^i_Y$
push forward to $\Omega^i_X(\log D_1)(-D_1)$
near the generic point of $D_1$. That is, we want
$\tr_f(\Omega^i_Y)$ to be contained in $(g,dg)\Omega^*_X$.
This follows from Proposition \ref{compatible},
using that $e=0\in k$. Thus we have constructed
the pushforward map in all cases.
\end{description}

We now return to the original morphism $f\colon Y\to X$
(before we restricted to an \'etale neighborhood of $X$).
The composition
of pullback and pushforward is multiplication by $\deg(f)$,
as one can check on an open subset where $X$ and $Y$ are smooth,
by the projection formula: for a form $\alpha$ on $Y$,
$\tr_f(f^*\alpha)=\tr_f(1)\alpha=\deg(f)\alpha$
\cite[Tag 0FLC]{Stacks}.

Tensoring the maps above with $O_X(A)$ and taking double duals, we have
pullback and pushforward maps
$$\Omega^{[i]}_X(\log D_X)(A-E_X)\to f_*(\Omega^{[i]}_Y(\log D_Y)(f^*A-E_Y))
\to \Omega^{[i]}_X(\log D_X)(A-E_X),$$
with composition equal to multiplication by $\deg(f)$, as we want. (Note that
$f^*A$ is a Weil divisor (with integer coefficients), because
$f$ is a finite surjective morphism between normal varieties.)
\end{proof}

\section{Endomorphisms and log Bott vanishing}

Kawakami and I showed that a projective variety with a suitable
endomorphism must satisfy Bott vanishing \cite[Theorem C]{KTot}.
(The endomorphism is assumed to be int-amplified and of degree invertible
in the base field.)
We now prove a logarithmic generalization, when the endomorphism
has a totally
invariant divisor. This generalizes Fujino's work
on the case of toric varieties \cite[Theorem 1.3]{Fujinolog}.
Indeed, every toric variety has an action of the multiplicative
monoid of positive integers,
and every toric divisor is totally invariant under
those endomorphisms.

\begin{theorem}
\label{endo-log-Bott}
Let $X$ be a normal projective variety over
a perfect field $k$.
Suppose that $X$ admits an int-amplified endomorphism $f$
whose degree is invertible in $k$. Let $D$ be a reduced divisor
on $X$ such that $f^{-1}(D)\subset D$. Then $(X,D)$ satisfies
log Bott vanishing. That is,
$$H^j(X,\Omega^{[i]}_X(\log D)(A-E))=0$$
for every reduced divisor $E$ with $0\leq E\leq D$,
$i\geq 0$, $j>0$, and $A$ an ample Weil divisor.
\end{theorem}

It may seem artificially strong to assume that an endomorphism
has a totally invariant divisor. But in fact, this property
comes up naturally in classifying varieties
with endomorphisms. See \cite{MZZ} or the proof
of Theorem \ref{endo-3-fold}.

Theorem \ref{endo-log-Bott}
is the most general vanishing property I know how to prove,
for a variety with an endomorphism.
The idea of allowing
any divisor $0\leq E\leq D$ was suggested by Fujino's result
for toric varieties \cite[Theorem 1.3]{Fujinolog}.

\begin{remark}
Bott vanishing can fail if we only assume that $f$
is int-amplified and separable (rather than of degree invertible in $k$),
by \cite[Proposition 5.1]{KTot}. A fortiori, log Bott vanishing
can fail in that situation. (Bott vanishing clearly fails
for inseparable endomorphisms, since {\it every }projective
variety over a finite
field has the Frobenius endomorphism, which is int-amplified.)
\end{remark}

\begin{proof}
(Theorem \ref{endo-log-Bott})
Since $f$ is int-amplified, there is an ample Cartier divisor $H$
on $X$ such that $f^*H-H$ is ample. In particular, $f^*H$
is ample, and so $f$ does not contract any curves. So
$f\colon X\to X$ is finite. Since $f^{-1}(D)\subset D$
and $f$ is surjective, we have $f^{-1}(D)=D$.

Since $f^{-1}(D)=D$, $f$ permutes the (finitely many) irreducible
components of $D$. After replacing $f$ by a positive iterate, we can assume
that $f$ maps each component
of $D$ to itself. So $f^{-1}(D_1)=D_1$, for each component $D_1$
of $D$. In particular, since $E$ is a reduced divisor
with $0\leq E\leq D$, we now have that $f^{-1}(E)=E$.
(In what follows, we use only that $f^{-1}(E)=E$, not that
$f$ maps each component of $D$ to itself. This makes the proof
clearer, in terms of notation.)

We are given that the degree of $f$ is invertible in $k$. The inverse
image of each irreducible component $D_1$ of $D$ is a single
irreducible component $D_2$ of $D$.
Therefore, the degree of $f\colon X\to X$
is the product of the degree of $f\colon D_2\to D_1$
and the ramification degree of $f$ along $D_2$. It follows that this
ramification degree is invertible in $k$ and that $k(D_2)$
is a finite separable extension of $k(D_1)$ via $f$. That is, $f$
is tamely ramified over each component of $D$.

Let $A$ be an ample Weil divisor on $X$.
By Lemma \ref{trace}, we have
pullback and pushforward maps
$$\Omega^{[i]}_X(\log D)(A-E)\to f_*(\Omega^{[i]}_X(\log D)(f^*A-E))
\to \Omega^{[i]}_X(\log D)(A-E),$$
with composition equal to multiplication by $\deg(f)$.
(Note that
$f^*A$ is a Weil divisor (with integer coefficients), because
$f$ is a finite surjective morphism between normal varieties.)

Given this, the proof of \cite[Theorem C]{KTot} works. Namely,
since $\deg(f)$ is invertible in $k$, it follows that
the pullback map is split injective as a map of $O_X$-modules. Let $j>0$.
Taking cohomology (and using that $f$ is finite), it follows
that $H^j(X,\Omega^{[i]}_X(\log D)(A-E))\to H^j(X,\Omega^{[i]}_X
(\log D)(f^*A-E))$ is split injective. The same argument works
for any iterate $f^e$ with $e\geq 1$.

Using that $f$ is int-amplified, we showed that
the iterates $(f^e)^*(A)$ become arbitrarily large in the ample
cone of $X$, as $e$ goes to infinity \cite[proof of Theorem C]{KTot}.
By Fujita vanishing for Weil divisors \cite[Lemma 2.6]{KTot},
it follows that there is an $e\geq 1$ such that $H^j(X,\Omega^{[i]}_X
(\log D)((f^e)^*A-E))=0$. By the previous paragraph,
we have $H^j(X,\Omega^{[i]}_X(\log D)(A-E))=0$,
as we want.
\end{proof}

\section{Morphisms}

Let $X$ and $Y$ be projective varieties of the same dimension
with Picard number 1.
If $X$ does not satisfy Bott vanishing, then there is an upper bound
on the degrees of all morphisms $Y\to X$ with degree
invertible in $k$ \cite[Proposition 2.7]{KTot}.
We now prove a log version of that result.

\begin{theorem}
\label{morphism-log-Bott}
Let $X$ and $Y$ be normal projective varieties of the same dimension
over a perfect field $k$. Assume that $X$ and $Y$ have Picard number 1.
Let $D_X\subset X$ and $B_Y\subset Y$
be reduced divisors. Suppose that there are morphisms
$f\colon Y\to X$ of arbitrarily high degree such that
the degree is invertible in $k$ and $f^{-1}(D_X)\subset B_Y$.
Then $(X,D_X)$ satisfies log Bott vanishing. That is, we have
$$H^j(X,\Omega^{[i]}_X(\log D_X)(A-E))=0$$
for every reduced divisor $0\leq E\leq D_X$, $i\geq 0$,
$j>0$, and $A$ an ample Weil divisor on $X$.
\end{theorem}

The proof of Theorem \ref{morphism-log-Bott}
shows that one can drop the assumption
that $X$ and $Y$ have Picard number 1 if one replaces
``$f$ of arbitrarily high degree'' by ``$f^*H$ arbitrarily large
in the ample cone of $Y$'', for a fixed ample Cartier divisor
$H$ on $X$.

\begin{proof}
Since each morphism $f\colon Y\to X$ has degree invertible
in $k$, the degree is positive.
Since $Y$ has Picard number 1,
the pullback of an ample Cartier divisor on $X$ is ample on $Y$,
and so $f$ is finite.

For each morphism $f\colon Y\to X$, let $D_Y$ be the sum of the components
of $f^{-1}(D_X)$ along which the ramification degree of $f$
is invertible in $k$, and likewise define $E_Y\subset f^{-1}(E_Y)$.
Since $D_Y$ and $E_Y$ are contained inside the fixed divisor $B_Y$,
we can assume (after passing to a subsequence of the morphisms $f$)
that $D_Y$ and $E_Y$ are the same for all the morphisms
$f\colon Y\to X$ that we consider.

Let $A$ be an ample Weil divisor on $X$. By Lemma \ref{trace},
we have pullback and pushforward maps
$$\Omega^{[i]}_X(\log D_X)(A-E)\to f_*(\Omega^{[i]}_Y(\log D_Y)(f^*A-E_Y))
\to \Omega^{[i]}_X(\log D_X)(A-E).$$
The composition is multiplication by $\deg(f)$, and so the pullback
map is split injective. It follows that $H^j(X,\Omega^{[i]}_X(\log D_X)(A-E))$
injects into $H^j(Y,\Omega^{[i]}_Y(\log D_Y)(f^*A-E_Y))$.
Since we have morphisms $f$ of arbitrarily large degree,
$f^*A$ becomes arbitrarily large in the ample cone of $Y$
(here just one ray). Therefore, for $f$ of sufficiently large degree,
Fujita vanishing for Weil divisors gives that
$H^j(Y,\Omega^{[i]}_Y(\log D_Y)(f^*A-E_Y))$ is equal to zero
\cite[Lemma 2.6]{KTot}.
It follows that $H^j(X,\Omega^{[i]}_X(\log D_X)(A-E))=0$,
as we want.
\end{proof}

\section{Images of toric varieties: log Bott vanishing}

A projective variety that is an image of a toric variety
(by a morphism of degree invertible in $k$)
must satisfy Bott vanishing \cite[Proposition 3.10]{KTot}.
We now prove a log version of that result.

\begin{theorem}
\label{toric-log-Bott}
Let $X$ be a normal projective variety
over a perfect field $k$. Suppose that there is a morphism $f$
from a proper toric variety $Y$ onto $X$.
If $Y\to Y_1\to X$ is the Stein factorization of $f$,
assume that the degree of $Y_1\to X$ is invertible in $k$.
Let $D_X\subset X$ be a reduced divisor such that
$f^{-1}(D_X)$ is a union of toric divisors. 
Then $(X,D_X)$ satisfies log Bott vanishing. That is, we have
$$H^j(X,\Omega^{[i]}_X(\log D_X)(A-E))=0$$
for every reduced divisor $0\leq E\leq D_X$, $i\geq 0$,
$j>0$, and $A$ an ample Weil divisor on $X$.
\end{theorem}

It may seem unrealistically strong to assume that $f^{-1}(D_X)$
is a union of toric divisors. Nonetheless,
the proof of Theorem \ref{image-3-fold} shows
how Theorem \ref{toric-log-Bott} can be used for classifying
images of toric varieties.

\begin{proof}
Every contraction of a toric variety is toric
\cite[Proposition 2.7]{Tanakakv}. (This was known earlier for projective
toric varieties \cite[Theorem 6.28
and exercise 7.2.3]{CLS}.)
So, replacing $Y$ by $Y_1$, we can assume that the surjection
$f\colon Y\to X$ is finite. By our assumptions, the degree
of $f$ is invertible in $k$. Let $D_Y$ be the sum of the
components of $f^{-1}(D_X)$ along which the ramification
degree $e$ of $f$ is invertible in $k$. 
Likewise, let $E_Y$ be the sum of the components of $f^{-1}(E)$
along which the ramification degree of $f$ is invertible in $k$.
By our assumptions, $D_Y$ and $E_Y$ are sums of toric divisors.

Let $A$ be an ample Weil divisor on $X$. By Lemma \ref{trace},
we have pullback and pushforward maps
$$\Omega^{[i]}_X(\log D_X)(A-E)\to f_*(\Omega^{[i]}_Y(\log D_Y)(f^*A-E_Y))
\to \Omega^{[i]}_X(\log D_X)(A-E),$$
with composition equal to multiplication by $\deg(f)$. 
Since $\deg(f)$ is invertible in $k$, it follows that
the pullback map is split injective as a map of $O_X$-modules. Let $j>0$.
Taking cohomology (and using that $f$ is finite), it follows
that $H^j(X,\Omega^{[i]}_X(\log D_X)(A-E))\to H^j(Y,\Omega^{[i]}_Y
(\log D_Y)(f^*A-E_Y))$ is split injective.
Since $f^*A$
is ample on $Y$, the latter cohomology group is zero
by log Bott vanishing for toric varieties \cite[Theorem 1.3]{Fujinolog}.
(Alternatively, this follows from Theorem \ref{endo-log-Bott},
using the multiplication maps on a toric variety.) It follows
that $H^j(X,\Omega^{[i]}_X(\log D_X)(A-E))=0$,
as we want.
\end{proof}

\section{Endomorphisms of del Pezzo surfaces
and Fano 3-folds}

\begin{theorem}
\label{endo-3-fold}
Let $X$ be a smooth Fano 3-fold over an algebraically closed
field $k$. 
If $X$ has an int-amplified endomorphism of degree
invertible in $k$, then $X$ is toric.
\end{theorem}

The proof uses Tanaka's theorem that the classification of smooth
Fano 3-folds has essentially the same form
in every characteristic \cite[Theorem 1.1]{Tanaka4},
\cite[Theorem 1.1]{Tanaka2}.
Without that, our proof applies to Fano 3-folds in any characteristic
that are given by the same construction
as one of the Fano 3-folds over $\C$
(classified by Iskovskikh and Mori-Mukai).

In characteristic zero, Theorem \ref{endo-3-fold}
was proved earlier by Meng, Zhang, and Zhong \cite[Theorem 1.4]{MZZ}.
Here we give a new proof which works in any characteristic.
For Fano 3-folds with Picard number 1, this is
\cite[Theorem A]{KTot}. We also prove an analogous
result for del Pezzo surfaces (Proposition \ref{endosurface}).

To prove Theorem \ref{endo-3-fold},
our basic idea is to use that $X$ satisfies Bott vanishing,
but that is not enough by itself: among the smooth
complex Fano 3-folds, 18 are toric while 19 others also
satisfy Bott vanishing \cite{TotaroBottFano}. In proving
Theorem \ref{endo-3-fold}, one key point is that
when $X$ has an endomorphism as above,
not only $X$ but also every contraction of $X$
satisfies Bott vanishing (Lemma \ref{bott-contraction}).

\begin{proof}
(Theorem \ref{endo-3-fold})
Assume that $X$ has an int-amplified endomorphism of degree
invertible in $k$. Then $X$ satisfies Bott vanishing
\cite[Theorem C]{KTot}. Among all smooth Fano 3-folds,
exactly 19 non-toric Fano 3-folds (up to isomorphism)
satisfy Bott vanishing
\cite[Theorem 0.1]{TotaroBottFano}. In Mori-Mukai's
numbering \cite{MM,MMerratum,IP,Fanography},
these are (2.26), (2.30), (3.15)--(3.16),
(3.18)--(3.24), (4.3)--(4.8), (5.1), and (6.1).
(To be precise, the answer is a subset of this in characteristic 2,
where only 9 non-toric Fano 3-folds on the known list
satisfy Bott vanishing \cite[section 2]{TotaroBottFano}.)
We need to show that none of these 19 varieties has an int-amplified
endomorphism $f$ of degree invertible in $k$.
Table \ref{nontoric} describes these 19 varieties, using information
from Mori-Mukai or the web site
Fanography \cite{MM, Fanography}. In the table, $V_5$ is the smooth quintic
del Pezzo 3-fold $\Gr(2,5)\cap\P^6\subset \P^9$, $Q$ is the smooth
quadric 3-fold, $W$ is the flag manifold $GL(3)/B$ (or equivalently,
a smooth divisor of degree $(1,1)$ in $\P^2\times \P^2$),
(3.17) is a smooth divisor of degree $(1,1,1)$ in $\P^1\times \P^1
\times \P^2$, and $S_5$ is the quintic del Pezzo surface.
\begin{table}
\centering
\begin{tabular}{c|p{0.8\textwidth}}
Fano 3-fold & Description \\
\hline
(2.26) & the blow-up of $V_5\subset\P^6$ along a general line in it \\
(2.30) & the blow-up of $Q\subset\P^4$ at a point \\
(3.15) & the blow-up of $Q$ along a disjoint line and conic \\
(3.18) & the blow-up of $Q$ along a conic and then along a fiber
in the exceptional divisor \\
(3.19) & the blow-up of $Q$ at 2 non-collinear points \\
(3.20) & the blow-up of $Q$ along 2 disjoint lines \\
(3.23) & the blow-up of $Q$ along a line and then along a fiber
of the exceptional divisor \\
(4.4) & the blow-up of $Q$ along a conic and then along 2 fibers
of the exceptional divisor \\
(5.1) & the blow-up of $Q$ along a conic and then along 3 fibers
of the exceptional divisor \\
(3.16) & the blow-up of $W\subset \P^2\times \P^2$ along a curve
of degree $(2,1)$ \\
(3.24) & the blow-up of $W$ along a curve of degree $(1,0)$ \\
(4.7) & the blow-up of $W$ along disjoint $(1,0)$ and $(0,1)$ curves \\
(4.3) & the blow-up of (3.17) in $\P^1\times \P^1\times \P^2$
along a curve of degree $(1,1,0)$ \\
(3.21) & the blow-up of $\P^1\times \P^2$ along a curve
of degree $(2,1)$ \\
(3.22) & the blow-up of $\P^1\times \P^2$ along a curve
of degree $(0,2)$ \\
(4.5) & the blow-up of $\P^1\times \P^2$ along disjoint 
$(2,1)$ and $(1,0)$ curves \\
(4.6) & the blow-up of $\P^3$ along 3 disjoint lines \\
(4.8) & the blow-up of $(\P^1)^3$ along a curve of degree
$(0,1,1)$ \\
(6.1) & $\P^1\times S_5$ \\
\end{tabular}
\caption{The 19 non-toric Fano 3-folds that satisfy Bott vanishing}
\label{nontoric}
\end{table}

For each of these 19 Fano 3-folds,
the nef cone of $X$ is the same as the (known)
nef cone in characteristic zero. In particular,
it is rational polyhedral. (The nef cone for every smooth
complex Fano 3-fold is given in \cite{CCGK}. The nef cone
was re-computed for the 19 varieties above
in any characteristic, yielding the same result
\cite{TotaroBottFano}. The argument also shows that all nef divisors
are semi-ample, and hence the contractions
of these varieties have the same description in every characteristic.)

\begin{lemma}
\label{bott-contraction}
Let $X$ be a normal projective variety over a perfect
field $k$. Suppose that $X$ has only finitely many contractions
(for example, this holds if $X$ is a Mori dream space,
or if the nef cone is rational polyhedral).
Assume that $X$ has an int-amplified endomorphism of degree
invertible in $k$. Then:
\begin{enumerate}[(i)]
\item Every contraction of $X$
admits an int-amplified endomorphism
of degree invertible in $k$.
\item Every contraction of $X$
satisfies Bott vanishing for ample Weil divisors.
\end{enumerate}
If we only assume that $X$ has a separable int-amplified endomorphism,
then every contraction of $X$ also has a separable int-amplified
endomorphism.
\end{lemma}

One can get around the assumption that $X$ has only
finitely many contractions by considering only contractions
that reduce the Picard number by 1 \cite[Definition 5.2
and Theorem 7.9]{MZnormal}.

\begin{proof}
Let $f\colon X\to X$ be a separable int-amplified endomorphism,
and let $\pi\colon X\to Z$
be a contraction.
Every curve on $X$ is the image of a curve under $f$.
It follows that the pullback linear map $f^*\colon N^1(X)\to N^1(X)$
is injective, hence an isomorphism.
It also follows that $f^*(\Nef(X))=\Nef(X)$.
A contraction is determined by a face of the cone of curves,
or equivalently by a face of the nef cone in $N^1(X)$;
so $f$ permutes the set of contractions of $X$. Since $X$
has only finitely many contractions, after replacing $f$
by a positive iterate, we can assume that $f$ preserves the
contraction $\pi$. That is, we have a commutative diagram
$$\xymatrix@C-10pt@R-10pt{
X \ar[r]^f \ar[d]^{\pi} & X\ar[d]^{\pi} \\
Z \ar[r]^g & Z,
}$$
for an endomorphism $g$ of $Z$. Here $g$ is separable since $f$ is.

An endomorphism $f$ of a normal projective variety over a field $k$
is int-amplified
if and only if all eigenvalues of $f^*$ on $N^1(X)_{\C}$
have absolute value greater than 1 \cite[Theorem 1.1]{Meng-amplified}.
(Meng assumes that $k$ has characteristic zero, but his proof
works in any characteristic.) In our situation,
the pullback $\pi^*\colon N^1(Z)\to N^1(X)$ is injective (because
every curve on $Z$ is the image of some curve on $X$). Since
all eigenvalues of $f^*$ on $N^1(X)_{\C}$ have absolute value
greater than 1, the same holds for the eigenvalues of $g^*$
on $N^1(Z)_{\C}$. That is, $g$ is an int-amplified
endomorphism of $Z$. That completes the proof for a separable
int-amplified endomorphism $f$.

Suppose in addition that $f$ has degree invertible in $k$.
The degree of $g$ divides the degree of $f$
and hence is invertible in $k$. It follows that $Z$ satisfies
Bott vanishing by \cite[Theorem C]{KTot} (generalized
as Theorem \ref{endo-log-Bott} above).
\end{proof}

This immediately rules out 13 of the 19 Fano 3-folds above,
since they contract to other varieties where Bott vanishing
fails, by Table \ref{nontoric}. Namely, by \cite[section 2]{TotaroBottFano}
(or earlier work), Bott vanishing fails
for the quintic del Pezzo 3-fold $V_5$,
the quadric 3-fold $Q$, the flag manifold $W=GL(3)/B$,
and the 3-fold (3.17), a smooth divisor in $\P^1\times \P^1\times
\P^2$ of degree $(1,1,1)$.

That leaves: (3.21), (3.22), (4.5),
(4.6), (4.8), and (6.1). To analyze these, we will use:

\begin{lemma}
\label{fibers}
Let $f\colon Y\to X$ be a finite surjective morphism
of normal varieties over an algebraically closed field $k$.
Let $\pi_X\colon X\to X_1$
and $\pi_Y\colon Y\to Y_1$ be contractions, and suppose
that there is a morphism $g\colon Y_1\to X_1$ making
the diagram commute:
$$\xymatrix@C-10pt@R-10pt{
Y \ar[r]^f \ar[d]^{\pi_Y} & X\ar[d]^{\pi_X} \\
Y_1 \ar[r]^g & X_1.
}$$
Let $\Delta_X$ be the set of points $P\in X_1(k)$ with
$\pi_X^{-1}(P)$ reducible, and likewise define
$\Delta_Y\subset Y_1(k)$. Then $g^{-1}(\Delta_X)\subset \Delta_Y$.
\end{lemma}

\begin{proof}
For each $k$-point $P$ in $Y_1$, I claim that
the fiber of $Y$ over $P$ maps
onto the fiber of $X$ over $g(P)$. This is fairly clear geometrically;
to be precise, we can imitate the proof of \cite[Lemma 7.3]{CMZ}.
Namely, suppose that $Y_P\to X_{g(P)}$ is not surjective.
Let $S=g^{-1}(g(P))-\{P\}$, a finite set. Then $S\neq \emptyset$
and $U:=Y-\pi_Y^{-1}(S)$ is an open dense subset of $Y$
not equal to $Y$, using that $Y$ is irreducible.
Since $X$ is normal, the finite surjection $f\colon Y\to X$ is an open map
\cite[Tag 0F32]{Stacks}, and so $f(U)$ is an open dense subset of $X$.
In particular, $f(U)\cap X_{g(P)}$ is open in $X_{g(P)}$.
Note that $f(U)=(X-X_{g(P)})\cup f(Y_P)$. So $f(Y_P)$ is open
in $X_{g(P)}$. It is not empty, since $Y_P$ is not empty.
Since $f$ is proper, $f(Y_P)$ is also closed
in $X_{g(P)}$. Since $\pi_X\colon X\to X_1$ is a contraction,
$X_{g(P)}$ is connected, and so $f(Y_P)=X_{g(P)}$, as we want.

An irreducible scheme cannot map onto a reducible scheme.
Therefore, $g^{-1}(\Delta_X)$ is contained in $\Delta_Y$.
\end{proof}

Let us continue the proof of Theorem \ref{endo-3-fold},
that a Fano 3-fold with an int-amplified endomorphism of degree
invertible in $k$ is toric.
Of the 6 remaining cases, we start with
(6.1),  which
is $\P^1$ times the quintic del Pezzo surface $X$.
By Lemma \ref{bott-contraction}, it suffices to show
that $X$ has no int-amplified endomorphism of degree
invertible in $k$. This was shown by Nakayama
\cite[Proposition 4.4]{Nakayama}.
We prove a bit more, namely that $X$ has no
separable int-amplified endomorphism.

\begin{proposition}
\label{endosurface}
A smooth del Pezzo surface over an algebraically closed field $k$
that admits a separable int-amplified endomorphism
must be toric.
\end{proposition}

\begin{proof}
The classification of smooth del Pezzo surfaces has the same form
in any characteristic \cite[section III.3]{Kollarrational}.
In particular, a smooth del Pezzo surface $X$ is toric if and only if
its degree $(-K_X)^2$ is at least 6. (Then $X$ is $\P^1\times \P^1$ or else
the blow-up of $\P^2$ at 3 or fewer points in general position.)
And every smooth del Pezzo surface
of degree at most 5 is a blow-up of the quintic del Pezzo
surface (the blow-up of $\P^2$ at 4 points in general
position). By Lemma \ref{bott-contraction},
it suffices to show that the quintic del Pezzo surface $X$
does not have a separable int-amplified endomorphism.

Suppose that the quintic del Pezzo surface $X$ has a separable
int-amplified endomorphism $f$.
Let $\pi_X$ be one of the contractions of $X$
to $\P^1$. By the proof of Lemma \ref{bott-contraction},
after replacing $f$ by a positive iterate,
the contraction $\pi_X$ is $f$-equivariant, giving 
a separable int-amplified endomorphism $g$ of $\P^1$.

The point is that the contraction $\pi_X$ from the quintic
del Pezzo surface to $\P^1$ has 3 singular fibers,
and these fibers are reducible (namely, the union
of two copies of $\P^1$). Let $\Delta_X\subset \P^1$ be the discriminant locus
of $\pi_X$, consisting of 3 points. By Lemma \ref{fibers},
$g^{-1}(\Delta_X)$ is contained in $\Delta_X$. Then pulling back
differential forms gives a map
$$\Omega^1_{\P^1}(\log \Delta_X)\to g_*\Omega^1_{\P^1}(\log
\Delta_X).$$

Equivalently, we can view this as a map of line bundles
from $g^*(K_{\P^1}+\Delta_X)$ to $K_{\P^1}+\Delta_X$, hence
as a section of $K_{\P^1}+\Delta_X-g^*(K_{\P^1}+\Delta_X)$.
Since $g$ is separable, this section is not identically zero.
But $K_{\P^1}+\Delta_X$ has degree 1, so $K_{\P^1}+\Delta_X
-g^*(K_{\P^1}+\Delta_X)$ has degree $1-\deg(g)<0$, a contradiction.
Thus the quintic del Pezzo surface does not have a separable
int-amplified endomorphism.
\end{proof}

Since the 3-fold (6.1) is $\P^1$ times the quintic
del Pezzo surface, Lemma \ref{bott-contraction}
and Proposition \ref{endosurface}
give that (6.1) does not admit a separable int-amplified
endomorphism. A fortiori, it does not have an int-amplified endomorphism
of degree invertible in $k$, as considered in Theorem
\ref{endo-3-fold}.

The proof for (4.8), the blow-up $X$
of $(\P^1)^3$ along a curve $F$ of degree $(0,1,1)$,
is somewhat similar. We can take $F$ to be a point in $\P^1$
times the diagonal $\Delta_{\P^1}^2$ in $(\P^1)^2$.
Suppose that $X$ has
an int-amplified endomorphism of degree invertible in $k$.
The contraction $\pi$ of $X$ to $(\P^1)^2$
(corresponding to the last two
$\P^1$ factors) has discriminant locus the diagonal
$\Delta_{\P^1}$, and the fibers over $\Delta_{\P^1}$ are reducible.
After replacing $f$ by an iterate, $\pi$ is $f$-equivariant,
by Lemma \ref{bott-contraction}.
Write $g\colon (\P^1)^2\to (\P^1)^2$ for the resulting
endomorphism of $(\P^1)^2$, which is int-amplified and
has degree invertible in $k$. By Lemma \ref{fibers},
$\Delta_{\P^1}$ is totally invariant under $g$.

By Theorem \ref{endo-log-Bott}, it follows
that $((\P^1)^2,\Delta_{\P^1})$
satisfies log Bott vanishing. In particular, taking $A=O(1,1)$
and $E=\Delta_{\P^1}\sim A$, we have
$H^1(X,\allowbreak \Omega^1_{(\P^1)^2}(\log \Delta_{\P^1}))=0$.
But in fact, this cohomology group is not zero. Indeed, for any smooth
divisor $D$ in a smooth variety $X$ over $k$, we have
an exact sequence
of coherent sheaves \cite[Tag 0FMW]{Stacks}:
$$0\to \Omega^1_X\to \Omega^1_X(\log D)
\xrightarrow[]{\text{Res}} O_D\to 0.$$
So we have an exact sequence of cohomology groups
$H^0(D,O_D)\to H^1(X,\Omega^1_X)\to H^1(X,\Omega^1_X(\log D))$.
For $X=(\P^1)^2$ and $D=\Delta_{\P^1}$, we have $h^0(D,O_D)=1$
and $h^1(X,\Omega^1_X)=2$, and so $H^1(X,\Omega^1_X(\log D))\neq 0$
as claimed.
This contradiction shows that
the Fano 3-fold (4.8) does not have an int-amplified
endomorphism of degree invertible in $k$.

Next, we exclude the 3-fold (4.6),
the blow-up of $\P^3$ along three pairwise disjoint lines $L_1,L_2,L_3$
over $k$.
Here $X$ has three contractions to $\P^1$, using that the blow-up
of $\P^3$ along each line $L_i$ is a $\P^2$-bundle over $\P^1$.
Let $\pi\colon X\to (\P^1)^2$ be the contraction given
by the morphisms to $\P^1$ associated to $L_1$ and $L_2$. For clarity,
first consider the blow-up $Y$ of $\P^3$ along $L_1$ and $L_2$;
then the contraction $Y\to (\P^1)^2$ is a $\P^1$-bundle.
The line $L_3\subset Y$ maps to a curve of degree $(1,1)$ in $(\P^1)^2$,
which we can take to be the diagonal $\Delta_{\P^1}$. Since $X$ is the blow-up
of $Y$ along $L_3$, the discriminant locus
of $\pi\colon X\to (\P^1)^2$ is the curve $\Delta_{\P^1}$.
The fibers of $\pi$ over that curve are reducible (the union
of two copies of $\P^1$).

If $X$ has an int-amplified
endomorphism of degree invertible in $k$,
then the pair $((\P^1)^2,\Delta_{\P^1})$ satisfies
log Bott vanishing. But this is false, as shown above.
So the Fano 3-fold (4.6) does not have an int-amplified
endomorphism of degree invertible in $k$.

Next, we rule out (3.22), the blow-up $X$
of $\P^1\times \P^2$ along a conic in $p\times \P^2$,
for a $k$-point $p$ in $\P^1$. The contraction of $X$
to $\P^2$ has discriminant locus a conic $F$ in $\P^2$.
As in the arguments above, if $X$ has an int-amplified
endomorphism of degree invertible in $k$,
then so does $\P^2$, and $F$ is totally invariant. Therefore,
the pair $(\P^2,F)$ satisfies log Bott vanishing. But this is false.
Namely, let $A=O(1)$ and $E=F\sim O(2)$; then we will show
that $H^1(\P^2,\Omega^1_{\P^2}(\log F)(A-E))=H^1(\P^2,
\Omega^1_{\P^2}(\log F)(-1))$ is not zero. Use the exact sequence
of coherent sheaves on $\P^2$:
$$0\to \Omega^1_{\P^2}\to \Omega^1_{\P^2}(\log F)
\xrightarrow[]{\text{Res}} O_F\to 0.$$
By the exact sequence $0\to \Omega^1_{\P^2}(-1)\to O(-2)^{\oplus 3}
\to O(-1)\to 0$ on $\P^2$, $\Omega^1_{\P^2}(-1)$ has zero cohomology
in all degrees. So we have an isomorphism
$$H^1(\P^2,\allowbreak \Omega^1_{\P^2}(\log F)(-1))
\cong H^1(F,\allowbreak O(-1)|_F).$$
Since the conic $F$ is isomorphic to $\P^1$ and $O(-1)$ has degree $-2$
on $F$, $h^1(F,O(-1)|_F)$ is 1, not 0. So log Bott vanishing
fails for $(\P^2,F)$. It follows that the 3-fold (3.22)
does not have
an int-amplified endomorphism of degree invertible in $k$.

The last cases are (3.21) and (4.5). These are handled by:

\begin{lemma}
\label{nodal}
The Fano 3-folds (3.21) and (4.5) contract to the quintic del Pezzo
3-fold with one node, which does not satisfy Bott vanishing.
\end{lemma}

\begin{proof}
Here (4.5) is a blow-up
of the Fano 3-fold of type (3.21). Next, the Fano 3-fold $X$
of type (3.21) is the blow-up of $\P^1\times \P^2$ along
a curve $F$ of degree $(2,1)$.
There is a contraction
from $X$ to the quintic del Pezzo 3-fold $Y$ with one node
\cite[section 5.4.2]{Prokhorov}. It remains to show that
$Y$ does not satisfy Bott vanishing. Indeed, I claim
that the Euler characteristic
$\chi(Y,\Omega^{[2]}_Y(1))$ is $-2<0$. This is harder
than previous cases because $Y$ is singular, but it is still manageable.

To prove this, it is convenient to know how the sheaf
$\Omega^2_Y$ is related to its reflexive hull $\Omega^{[2]}_Y$,
for a 3-fold $Y$ with a node. More generally,
let $Y$ be a normal hypersurface
in a smooth variety $W$ with $\dim(W)=n+1$,
so that locally $Y$ is the zero locus
of a regular function $g$. 
Define the {\it torsion }and {\it cotorsion }of a coherent sheaf
$M$ to be the kernel and cokernel of the natural map
$M\to M^{**}$.
Consider the complex $K$:
$$0\to \Omega^0_W|_Y\to \Omega^1_W|_Y\to\cdots\to
\Omega^{n+1}_W|_Y\to 0,$$
with differentials given by $\wedge dg$. Graf described
the torsion and cotorsion of the sheaves $\Omega^j_Y$
in terms of this complex: we have
$H^j(K)\cong \tor \Omega^j_Y$ and $H^j(K)\cong \cotor \Omega^{j-1}_Y$
\cite[Theorem 1.11]{Graf}. 
For $Y$ the 3-fold node, we can take
$W=A^4$ and $g=xy-zw$. The complex above is the Koszul complex
for the regular sequence $\partial g/\partial x=y$,
$\partial g/\partial y=x$,
$\partial g/\partial z=-w$,
$\partial g/\partial w=-z$, tensored over $O_W$ with $O_Y$.
These four functions generate the ideal of the origin $P$ in $W=A^4$;
so $\tor \Omega^i_Y\cong \Tor_{4-i}^{O_W}(O_Y,O_P)$
and $\cotor \Omega^i_Y\cong \Tor_{3-i}^{O_W}(O_Y,O_P)$.
These Tor groups are easy to compute, using the free resolution
$0\to O_W\xrightarrow[]{g} O_W\to O_Y\to 0$ of $O_Y$ as an $O_W$-module.
Namely, it follows that $\Tor_*^{O_W}(O_Y,O_P)$
is the homology of the complex
$$0\to O_P \xrightarrow[]{0} O_P\to 0,$$
which is a 1-dimensional $k$-vector space in degrees 1 and 0, and zero
in other degrees.
Therefore, $\Omega^1_Y$ is reflexive,
while $\Omega^2_Y$ is torsion-free and its cotorsion
is a 1-dimensional vector space, supported 
at the node $P$. That is, we have a short exact sequence
of sheaves on a nodal 3-fold $Y$:
$$0\to \Omega^2_Y\to \Omega^{[2]}_Y\to O_P\to 0.$$

Hence, for $Y$ the quintic del Pezzo 3-fold with one node,
we have $\chi(Y,\Omega^{[2]}_Y(1))=1+\chi(Y,\Omega^2_Y(1))$.
It remains to show that $\chi(Y,\Omega^2_Y(1))=-3$.

For a nodal 3-fold $Y$ in a smooth 4-fold $W$ over $k$
(as for any effective Cartier
divisor), we have an exact
sequence of coherent sheaves on $Y$ \cite[Tag 00RU]{Stacks}:
$$O(-Y)|_Y\to \Omega^1_W|_Y\to \Omega^1_Y\to 0.$$
Taking exterior powers, we have an exact sequence
$\Omega^{j-1}_Y(-Y)\to \Omega^j_W|_Y\to \Omega^j_Y\to 0$
for any $j$. For $j=2$, the map $\Omega^1_Y(-Y)\to \Omega^2_W|_Y$
is clearly injective
outside the node. Since $\Omega^1_Y(-Y)$ is torsion-free (by the analysis
above), this map is actually injective as a map of sheaves on
all of $Y$.
That is, we have an exact sequence $0\to\Omega^1_Y(-Y)\to \Omega^2_W|_Y
\to \Omega^2_Y\to 0$. Likewise, we have an exact sequence
$0\to O_Y(-Y)\to \Omega^1_W|_Y\to \Omega^1_Y\to 0$ on $Y$.
Finally, we can tensor the exact sequence
$0\to O_W(-Y)\to O_W\to O_Y\to 0$ with any vector bundle on $W$,
giving $0\to \Omega^1_W(-Y)\to \Omega^1_W\to \Omega^1_W|_Y\to 0$
and $0\to \Omega^2_W(-Y)\to \Omega^2_W\to \Omega^2_W|_Y\to 0$.

Apply this to the quintic del Pezzo 3-fold $Y\subset \P^6$ with a node,
which is a codimension-3 linear section
of the Grassmannian $\Gr(2,5)\subset \P^9$. The quintic del Pezzo 3-fold
with one node is unique up to isomorphism,
with explicit equations labelled $X_{5,2,4}$ in \cite[section A.1.6]{KP}.
By the equations, $Y$ is a hyperplane
section of a smooth codimension-2 linear section $W\subset \P^7$.
By the previous paragraph, we can rewrite $\chi(Y,\Omega^2_Y(1))$
in terms of Euler characteristics on $W$; but this formula
would be exactly the same with $Y$ replaced by a smooth hyperplane
section $V_5\subset \P^6$. Therefore, $\chi(Y,\Omega^2_Y(1))=
\chi(V_5,\Omega^2_{V_5}(1))=-3$ \cite[section 2]{TotaroBottFano}.
As discussed above, it follows that $\chi(Y,\Omega^{[2]}_Y(1))=-2$,
and so $Y$ does not satisfy Bott vanishing. Lemma \ref{nodal}
is proved.
\end{proof}

It follows that neither (3.21) nor (4.5)
admits an int-amplified endomorphism
of degree invertible in $k$.
Theorem \ref{endo-3-fold} is proved.
\end{proof}

\section{Images of toric varieties: del Pezzo surfaces
and Fano 3-folds}

Occhetta and Wi\'sniewski conjectured that a smooth complex
projective variety $X$ that admits a surjective morphism from a proper
toric variety must be toric \cite{OW}. (This is known for contractions,
and so it suffices to consider a finite surjective morphism.) 
Occhetta-Wi\'sniewski's conjecture was proved
by Achinger, Witaszek, and Zdanowicz for $X$ of dimension at most 2,
and also for $X$ a Fano 3-fold \cite[proof of Theorem 4.4.1]{AWZ1},
\cite[Theorems 6.9 and 7.7]{AWZ2}. Using Bott vanishing,
we will reprove this result for Fano 3-folds
and extend it to positive characteristic
(Theorem \ref{image-3-fold}). Kawakami and I proved this extension earlier
for Fano 3-folds with Picard number 1
\cite[Proposition 3.10]{KTot}.

Another approach to Theorem \ref{image-3-fold} appeared
after the first version of this paper. Namely, Kawakami
and Takamatsu used Lemma \ref{trace} above to show that the image
of an $F$-liftable variety in characteristic $p$ by a morphism
of degree prime to $p$ is $F$-liftable \cite[Corollary 3.13]{KTak}.
In particular, the image of a toric variety by a morphism
of degree prime to $p$ is $F$-liftable. But Achinger-Witaszek-Zdanowicz
showed that a smooth Fano 3-fold that is $F$-liftable must be toric.
So a Fano 3-fold that is an image of a toric variety by a morphism
of degree prime to $p$ must be toric.

We will prove Theorem \ref{image-3-fold} using log Bott vanishing
rather than $F$-liftability.
Thus our method for analyzing images of toric varieties
closely parallels our approach to finding which varieties
have nontrivial endomorphisms (Theorem \ref{endo-3-fold}).

We start with the analogous result in dimension 2.

\begin{proposition}
\label{toricsurface}
A smooth del Pezzo surface over an algebraically closed field $k$
that is the image of a proper toric variety must be toric.
\end{proposition}

Surprisingly, Proposition \ref{toricsurface} works in any characteristic
with no assumption on the degree of the morphism $f$; $f$ may even
be inseparable. The proof does not use Bott vanishing.

\begin{proof}
A smooth del Pezzo surface $X$ is toric if and only if
its degree $(-K_X)^2$ is at least 6.
And every smooth del Pezzo surface
of degree at most 5 is a blow-up of the quintic del Pezzo
surface (the blow-up of $\P^2$ at 4 points in general
position). So it suffices to show that there is no surjection $f$
from a proper toric variety $Y$ to the quintic del Pezzo
surface $X$.

Let $Y\to Z\to X$ be the Stein
factorization of $f$; then $Z$ is also toric, since
every contraction of a toric variety is toric
\cite[Proposition 2.7]{Tanakakv}.
Replacing $Y$ by $Z$, we can assume
that $f\colon Y\to X$ is finite and surjective.

Let $\pi_X$ be one of the contractions of $X$
to $\P^1$. Let $Y\to Y_1\to \P^1$ be the Stein factorization
of the composition $Y\to X\to \P^1$, where we write $\pi_Y\colon Y\to Y_1$
and $g\colon Y_1\to \P^1$. Then $Y_1$ is a proper
toric curve, hence isomorphic to $\P^1$, and $g
\colon Y_1\to \P^1$ is finite and surjective.

As we used in the proof of Proposition \ref{endosurface},
the contraction $\pi_X$ from the quintic
del Pezzo surface to $\P^1$ has 3 singular fibers,
and these fibers are reducible (namely, the union
of two copies of $\P^1$). Let $\Delta_X\subset \P^1$ be the discriminant locus
of $\pi_X$, consisting of 3 points. By Lemma \ref{fibers},
$g^{-1}(\Delta_X)$ is contained in the discriminant locus
$\Delta_Y$ of $\pi_Y\colon Y\to Y_1$.
But $Y\arrow Y_1$ is a toric morphism, so its discriminant
locus is contained in the toric divisor of $Y_1$, which consists
of 2 points. Since $g\colon Y_1\to \P^1$ is surjective, this is
a contradiction (2 points cannot map onto 3 points).
We have shown that
the quintic del Pezzo surface is not the image of a proper toric variety.
\end{proof}

\begin{theorem}
\label{image-3-fold}
Let $X$ be a smooth Fano 3-fold over an algebraically closed
field $k$.
If $X$ is the image of a proper toric 3-fold by a morphism
of degree invertible in $k$, then $X$ is toric.
\end{theorem}

As in Theorem \ref{endo-3-fold},
the proof uses Tanaka's theorem that the classification of smooth
Fano 3-folds has essentially the same form
in every characteristic \cite[Theorem 1.1]{Tanaka4},
\cite[Theorem 1.1]{Tanaka2}.
Without that, our proof applies to Fano 3-folds in any characteristic
that are given by the same construction
as one of the Fano 3-folds over $\C$
(classified by Iskovskikh and Mori-Mukai).

\begin{proof}
Assume that $X$ is the image of a proper toric 3-fold $Y$
by a morphism $f$ of degree invertible in $k$.
After replacing $f$ by its Stein factorization,
we can also assume that $f\colon Y\to X$ is finite. Here $Y$ is still
toric, because
every contraction of a toric variety is toric
\cite[Proposition 2.7]{Tanakakv}.
Then $X$ satisfies Bott vanishing, by Theorem
\ref{toric-log-Bott}.
There are exactly 19 non-toric
Fano 3-folds (up to isomorphism)
that satisfy Bott vanishing
\cite[Theorem 0.1]{TotaroBottFano}. In Mori-Mukai's
numbering, these are (2.26), (2.30), (3.15)--(3.16),
(3.18)--(3.24), (4.3)--(4.8), (5.1), and (6.1).
(To be precise, the answer is a subset of this in characteristic 2,
where only 9 non-toric Fano 3-folds on the known list
satisfy Bott vanishing \cite[section 2]{TotaroBottFano}.)
We need to show that none of these 19 varieties can be
the image of a toric variety $Y$ by a finite morphism $f$
of degree invertible in $k$.

Of these 19 varieties, 15 have contractions that
do not satisfy Bott vanishing, by the proof of Theorem
\ref{endo-3-fold}, including Lemma \ref{nodal}.
This excludes all the cases
except four: (3.22), (4.6), (4.8), and (6.1).

Here (6.1) is $\P^1$ times the quintic del Pezzo surface.
That surface is not an image of a toric variety,
by Proposition \ref{toricsurface}. So (6.1) is not the image
of a toric variety.

Next, we rule out (3.22), the blow-up $X$
of $\P^1\times \P^2$ along a conic in $q\times \P^2$,
for a $k$-point $q$ in $\P^1$. The contraction of $X$
to $X_1:=\P^2$ has discriminant locus a conic $F\subset \P^2$,
and the fibers over points of $F$ are reducible (a union
of two $\P^1$'s).
Suppose that there is a proper toric 3-fold $Y$ over $k$
with a morphism $f\colon Y\to X$ such that $\deg(f)$
is invertible in $k$.

Let $Y\to Y_1\to X_1$ be the Stein factorization
of the composition $Y\to X\to X_1$:
$$\xymatrix@C-10pt@R-10pt{
Y \ar[r]_{f} \ar[d] & X\ar[d] \\
Y_1 \ar[r] & X_1.
}$$
Then $Y_1$ is a proper
toric surface over $k$, and $h\colon Y_1\to X_1$ is finite, of degree
invertible in $k$. The discriminant locus of $Y\to Y_1$ is
contained in the union of the toric divisors.
By Lemma \ref{fibers}, $h^{-1}(F)$ is contained
in the discriminant locus of $Y\to Y_1$, hence
in the union of the toric divisors in $Y_1$.

By Theorem \ref{toric-log-Bott}, it follows
that the pair $(X_1,F)=(\P^2,F)$ satisfies log Bott vanishing.
But that is false, by the proof of Theorem \ref{endo-3-fold}.
Thus (3.22) is not the image of a toric 3-fold by a morphism
of degree invertible in $k$.

A similar argument rules out (4.8),
the blow-up $X$
of $(\P^1)^3$ along a curve of degree $(0,1,1)$.
The contraction of $X$ to $X_1:=(\P^1)^2$ (corresponding to the last two
$\P^1$ factors) has discriminant locus a curve $F$ of degree $(1,1)$,
which we can take to be the diagonal
$\Delta_{\P^1}$ in $(\P^1)^2$. The fibers over points in $F$
are reducible (a union of two $\P^1$'s).
Let $Y\to Y_1\to X_1$ be the Stein factorization
of the composition $Y\to X\to X_1$:
$$\xymatrix@C-10pt@R-10pt{
Y \ar[r]_{f} \ar[d] & X\ar[d] \\
Y_1 \ar[r] & X_1.
}$$
Then $Y_1$ is a proper
toric surface over $k$, and $h\colon Y_1\to X_1$ is finite, of degree
invertible in $k$. 

By the same argument as in the previous case, $h^{-1}(F)$
must be a toric divisor in $Y_1$.
By Theorem \ref{toric-log-Bott}, it follows
that the pair $(X_1,F)=((\P^1)^2,\Delta_{\P^1})$ satisfies log Bott vanishing.
But that is false, by the proof of Theorem \ref{endo-3-fold}.
Thus (4.8) is not the image of a toric 3-fold by a morphism
of degree invertible in $k$.

Finally, let $X$ be the Fano 3-fold (4.6),
the blow-up of $\P^3$ along three pairwise disjoint lines $L_1,L_2,L_3$
over $k$.
Suppose that there is a proper toric 3-fold $Y$ over $k$
with a morphism $f\colon Y\to X$ such that $\deg(f)$
is invertible in $k$.
By the proof of Theorem \ref{endo-3-fold},
$X$ has a contraction to $X_1:=(\P^1)^2$ with
discriminant locus the diagonal $F:=\Delta_{\P^1}$
in $(\P^1)^2$. The fibers over $F$ are reducible (the union
of two $\P^1$'s). Then the same argument in the previous case
(using that $((\P^1)^2,\Delta_{\P^1})$ does not satisfy log Bott vanishing)
yields a contradiction. So
the Fano 3-fold (4.6) is not the image
of a toric 3-fold by a morphism
of degree invertible in $k$.
\end{proof}

\section*{Acknowledgements}

This work was supported by NSF grant DMS-2054553, the
National Center for Theoretical Sciences (NCTS),
Simons Foundation grant SFI-MPS-SFM-00005512,
and the Charles Simonyi Endowment
at the Institute for Advanced Study.
Thanks to Louis Esser, Tatsuro Kawakami, Supravat Sarkar,
and the referee.

\bibliography{endo.bib}
\bibliographystyle{alpha}

\end{document}